\input ./style/arxiv-general.cfg
\documentclass[aop,MSNbibl,dvips]{arximspdf}
\makeatletter
   \@ifpackageloaded{graphicx}{}{\usepackage{graphicx}}
\makeatother

\doi{10.1214/14-AOP940} 
\volume{43}
\issue{5}
\pubyear{2015}
\firstpage{2458}
\lastpage{2480}
\docsubty{FLA}

\makeatletter
\newcommand{\rrvert}{\vert}
\newcommand{\llvert}{\vert}
\newtheorem{theorem}{Theorem}[section]
\newtheorem{lemma}[theorem]{Lemma}
\newtheorem{proposition}[theorem]{Proposition}
\newproclaim{definition}[theorem]{Definition}
\newtheorem{corollary}[theorem]{Corollary}
\def\bC{{\overline C}}
\def\bE{{\overline E}}
\def\bA{{\overline A}}
\def\bB{{\overline B}}
\def\ta{\mbox{two-arms}}
\makeatother

\begin{document}
\begin{frontmatter}

\title{A lower bound on the two-arms exponent for critical percolation
on the lattice}
\runtitle{Two-arms exponent for critical percolation}

\begin{aug}
\author[A]{\fnms{Rapha\"el}~\snm{Cerf}\corref{}\ead[label=e1]{rcerf@math.u-psud.fr}\ead[label=u1,url]{http://www.foo.com}}
\runauthor{R. Cerf}
\affiliation{Universit\'e Paris Sud and IUF}
\address[A]{Universit\'e Paris Sud\\
Math\'ematique, B\^atiment~$425$\\
91405 Orsay Cedex\\
France\\
\printead{e1}} 
\end{aug}

\received{\smonth{7} \syear{2013}}
\revised{\smonth{5} \syear{2014}}

%
\begin{abstract}
We consider the standard site percolation model on the $d$-dimensional lattice.
A direct consequence of the
proof
of the uniqueness of the infinite cluster of Aizenman, Kesten and Newman
[\textit{Comm. Math. Phys.} \textbf{111} (1987) 505--531]
is that the two-arms exponent is larger than or equal to $1/2$.
We improve slightly this lower bound in any dimension $d\geq2$.
Next, starting only with the hypothesis that $\theta(p)>0$, without
using the slab technology, we derive
a quantitative estimate establishing long-range order in a finite box.
\end{abstract}

%
\end{frontmatter}

\section{Introduction}\label{sec1}
We consider the site percolation model on ${\mathbb Z}^d$.
Each site is declared open with probability~$p$
and closed with probability~$1-p$, and the sites
are independent.
Little is rigorously known on the percolation model
at the critical point~$p_c$ in three dimensions.
Barsky, Grimmett and Newman have proved
that there is no percolation at the critical point
in a half-space. Grimmett and Marstrand have proved that the critical
points in a half-space and in the full space coincide.
A~full account of these results and their proofs can be found
in Grimmett's book~\cite{GR}.
Kesten's book presents also some estimates valid at the critical point
(see Chapter~5 of \cite{KE}).
There exists one remarkable result, a rigorous lower bound
on the two-arms exponent, which says that,
for any $d\geq2$,
%
\[
\exists\kappa>0, \forall n\geq1\qquad P_{p_c} \bigl(\ta(0,n) \bigr)
\leq\frac{\kappa\ln n}{\sqrt{n}}.
\]
The event ``$\ta(0,n)$'' is the event that two neighbors of $0$
are connected to the boundary of the box
$\Lambda(n)=[-n,n]^d$ by two disjoint open clusters.
Although some percolationists are aware of this estimate
(e.g., it is explicitly used by Zhang in \cite{ZH}),
it does not seem to be fully written in the literature.
This estimate can be obtained as a byproduct of the proof
of the uniqueness of the infinite cluster of Aizenman, Kesten and Newman
\cite{AKN}.
This deep proof was originally written for a quite general percolation model.
A simplified and illuminating version has been worked out by
Gandolfi, Grimmett
and Russo \cite{GGR}.
The two-arms estimate
is
obtained by taking $\varepsilon=\kappa\ln n/\sqrt{n}$ in the proof
of \cite{GGR}.
Nowadays the uniqueness of the infinite cluster in percolation is proved
with the help of the more robust Burton--Keane argument; see, for instance,
\cite{GR} or \cite{HJ}. Yet the Burton--Keane argument relies on translation
invariance,
and it does not yield any quantitative estimate,
contrary to the
argument of Aizenman, Kesten and Newman.
The first main result of this paper is a slightly improved lower bound
on the two-arms exponent.

%
\begin{theorem}\label{fm}
Let $d\geq2$ and let $p_c$ be the critical probability of
the site percolation model in $d$ dimensions.
We have
\[
\limsup_{n\to\infty} \frac{1}{\ln n} \ln P_{p_c} \bigl(
\ta(0,n) \bigr) \leq-\frac{2d^2+3d-3}{4d^2+5d-5}.
\]
\end{theorem}

%
\begin{figure}

\includegraphics{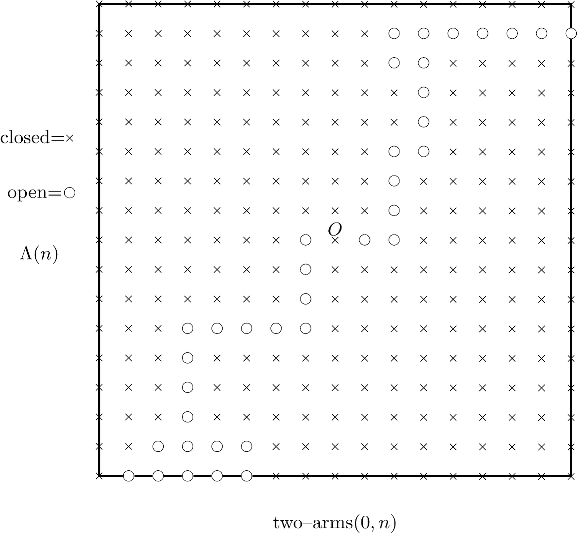}

\label{fig1}
\end{figure}

In two dimensions, our two-arms event corresponds to a four-arms event
with alternating colors. The corresponding exponent is rigorously known
to be equal to $5/4$ for site percolation on the triangular lattice
(see \cite{SW}), and our lower bound is $11/21$.
In high dimensions, the exponent associated to two arms is rigorously proven
to be equal to four \cite{KN}, and our lower bound converges to $1/2$
as the dimension
goes to infinity.
In three dimensions, we obtain
the following estimate:
\[
\forall\gamma<\frac{12}{23}, \exists c>0, \forall n\geq1\qquad
P_{p_c} \bigl(\ta(0,n) \bigr) \leq\frac{c}{n^\gamma}.
\]
%
To prove Theorem~\ref{fm},
we rework the proof of \cite{GGR} in order to obtain
an inequality of the form
\[
P_{p_c} \bigl(\ta(0,n) \bigr) \leq\frac{2d\ln n}{\sqrt{ \llvert \Lambda(n)
\rrvert }} E \bigl( \sqrt{
\llvert { \mathcal C}\rrvert } \bigr)+\mbox{negligible term},
\]
where ${\mathcal C}$ is the collection of the clusters joining $\Lambda
(n)$ to
the boundary of $\Lambda(2n)$.
From this inequality, we obtain the previously known estimate on the
two-arms event by bounding the number of clusters in ${\mathcal C}$
by $2d(2n+1)^{d-1}$.
We next try to enhance the control on the number of clusters. It turns out
that the expectation of this number can be bounded with the help
of the probability of the two-arms event.
Our strategy consists in controlling the two-arms event associated to a box.
This is the purpose of our second main result.
The event ``$\ta(\Lambda(n),n^\alpha)$''
is the event that two sites of the box
$\Lambda(n)=[-n,n]^d$
are connected to the boundary of the box
$\Lambda(n+n^\alpha)$ by two disjoint open clusters.

%
\begin{theorem}\label{sm}
Let $d\geq2$ and let $p_c$ be the critical probability of
the site percolation model in $d$ dimensions.
Let $\alpha$ be such that
\[
\alpha> \frac{2d^2+2d-2}{2d^2+3d-3} \bigl(4d^2+5d-5 \bigr).
\]
We have
\[
\lim_{n\to\infty} P_{p_c} \bigl(\ta \bigl(
\Lambda(n),n^\alpha \bigr) \bigr) = 0.
\]
\end{theorem}

For $d=3$, this gives
\[
\lim_{n\to\infty} P_{p_c} \bigl(\ta \bigl(
\Lambda(n),n^{43} \bigr) \bigr) = 0.
\]
%
Next, we cover the boundary of the box
$\Lambda(n)=[-n,n]^d$ by a collection of boxes of side length $n^\beta$,
with $\beta$ small. Theorem~\ref{sm} yields an estimate on
the number of small boxes joined to the boundary of
$\Lambda(2n)$ by at most one cluster, from which we obtain an upper bound
on the mean number of open clusters joining
$\Lambda(n)$
to the internal vertex boundary of
$\Lambda(2n)$.
This gives an upper bound on $E(|{\mathcal C}|)$ in terms of the
two-arms event.
This way we obtain an inequality of the form
\[
P \bigl(\ta(0,3n) \bigr) \leq\frac{
c'
\ln n}{\sqrt{n}} \biggl(\frac{1}{k^{d-1}} +
k^{2d^2+2d-2} P \bigl(\ta(0,n) \bigr) \biggr)^{1/2}. %
\]
Iterating this inequality with an adequate choice of $k\leq n$,
we progressively improve the exponent $1/2$.
We obtain a sequence of exponents converging geometrically toward
the limiting value presented in Theorem~\ref{fm}.
The final improvement is quite disappointing and the value is probably
quite far from the correct one.

%
\begin{figure}

\includegraphics{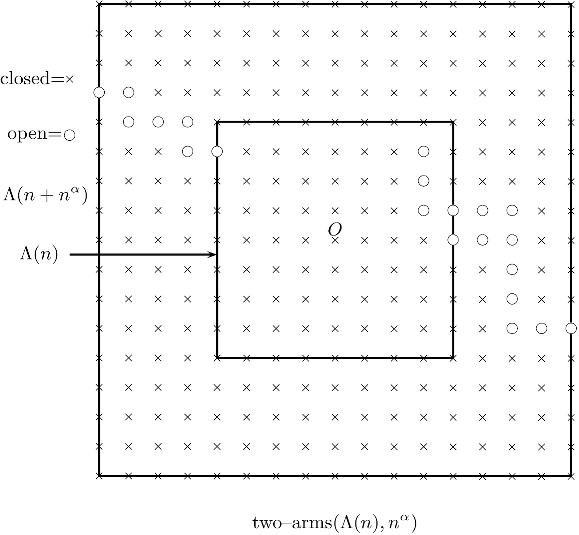}

\label{fig2}
\end{figure}

Our third main result is a little minor step for the establishment
of long-range order in a finite box. This is a central question, which
if correctly answered, should lead to a proof that $\theta(p_c)=0$.
For $\Lambda$ a box and $x,y$ in $\Lambda$, we denote
by \mbox{$\{ x\longleftrightarrow y\mbox{ in }\Lambda\}$}
the event that $x,y$ are connected by an open path inside $\Lambda$.

\begin{theorem}\label{tm}
Let $d\geq2$ and let $p$ be such that $\theta(p)>0$.
Let $\alpha$ be such that
\[
\alpha> \frac{4d^2+5d-5}{2d^2+3d-3}(3d-1).
\]
We have
\[
\inf_{n\geq1} \inf \bigl\{ P_{p} \bigl(x
\longleftrightarrow y\mbox{ in }\Lambda \bigl(n^\alpha \bigr) \bigr)\dvtx
x,y \in\Lambda(n) \bigr\} > 0.
\]
\end{theorem}

For $d=3$, this gives the following estimate:
%
\[
\exists\rho>0, \forall n\geq1, \forall x,y\in\Lambda(n)\qquad
P_{p} \bigl(x\longleftrightarrow y\mbox{ in }\Lambda
\bigl(n^{16} \bigr) \bigr) \geq\rho.
\]
One of the most important problems in percolation is to prove that, in
three dimensions, there is no infinite cluster at the critical point.
The most promising strategy so far seems to perform a renormalization
argument \cite{GR,PI}. The missing ingredient is a suitable construction
helping to define a good block, starting solely with the hypothesis
that $\theta(p)>0$. For instance, it would be enough to have the
above estimate within a box of side length proportional to $n$.
Moreover, if the famous conjecture $\theta(p_c)=0$ was true, such
an estimate would indeed hold.
Here again, we are still far from the desired result.
Our technique to prove Theorem~\ref{tm} is to inject the hypothesis
$\theta(p)>0$ inside the proof of the two-arms estimate for a box.
This allows to obtain a much better control on the probability of a
long connection, which is unfortunately still far from optimal.
\section{Basic notation}
Two sites $x,y$ of the lattice ${\mathbb Z}^d$ are said to be connected
if they
are nearest neighbors, that is, if $|x-y|=1$.
Let $A$ be a subset of ${\mathbb Z}^d$. We define
its internal boundary
$\partial^{\mathrm{in}}A$
and its external boundary
$\partial^{\mathrm{out}}A$ by
\begin{eqnarray*}
\partial^{\mathrm{in}}A &=& \bigl\{ x\in A\dvtx \exists y\in A^c, \llvert x-y\rrvert =1 \bigr\},
\\
\partial^{\mathrm{out}}A &=& \bigl\{ x\in A^c\dvtx \exists y\in A, \llvert x-y\rrvert =1 \bigr\}.
\end{eqnarray*}
For $x\in{\mathbb Z}^d$, we denote by $C(x)$ the open cluster
containing~$x$, that is, the connected component of the
set of the open sites containing~$x$.
If $x$ is closed, then
$C(x)$ is empty.
For $n\in{\mathbb N}$, we denote by $\Lambda(n)$ the cubic box
\[
\Lambda(n) = [-n,n]^d.
\]
%
Let $n,\ell$ be two integers.
We consider
the open clusters of the percolation
configuration restricted to $\Lambda(n+\ell)$.
These open clusters are the connected components of the graph
having for vertices the sites of $\Lambda(n+\ell)$ which are open,
endowed with edges between nearest neighbors.
We denote by ${\mathcal C}$ the collection
of the open clusters in $\Lambda(n+\ell)$ which intersect both
$\Lambda(n)$ and $\partial^{\mathrm{in}}\Lambda(n+\ell)$, that is,
\[
{\mathcal C}= \bigl\{ C\mbox{ open cluster in } \Lambda(n+\ell)\dvtx C\cap
\Lambda(n) \neq\varnothing, C\cap\partial^{\mathrm{in}}\Lambda(n+\ell) \neq
\varnothing \bigr\}.
\]
%
\section{The proof of Gandolfi, Grimmett and Russo}
We reproduce here the initial step of the
argument of
Gandolfi, Grimmett
and Russo to prove the uniqueness of
the infinite cluster \cite{GGR}.
This argument was obtained from the more complex work of
Aizenman, Kesten and Newman \cite{AKN}.
The only difference is that we introduce an additional parameter~$\ell$.
We will use specific values for $\ell$ later on.
We define the following three subsets of $\Lambda(n)$:
\begin{eqnarray*}
F &=& \bigcup_{C\in{\mathcal C}}C\cap\Lambda(n), \qquad G =
\bigcup_{C\in{\mathcal C}}\partial^{\mathrm{out}}C\cap \Lambda(n),
\\
H &=& \bigcup_{C_1,C_2\in{\mathcal C}, C_1\neq C_2} \bigl( \partial^{\mathrm{out}}C_1
\cap\partial^{\mathrm{out}}C_2 \cap\Lambda(n) \bigr).
\end{eqnarray*}
A site of $\Lambda(n)$ belongs to $F$ if it is connected to
$\partial^{\mathrm{in}}\Lambda(n+\ell)$ by an open path.
A~site of $\Lambda(n)$ belongs to $G$ if it is closed and it
has a neighbor which is connected to
$\partial^{\mathrm{in}}\Lambda(n+\ell)$ by an open path.
A site of $\Lambda(n)$ belongs to $F\cup G$ if it
has a neighbor which is connected to
$\partial^{\mathrm{in}}\Lambda(n+\ell)$ by an open path.
Yet, for any $x\in\Lambda(n)$, the~event
\[
\bigl\{ \mbox{a neighbor of $x$ is connected to $\partial^{\mathrm{in}}
\Lambda(n+\ell)$ by an open path} \bigr\}
\]
is independent of the status of the site~$x$ itself, therefore,
\begin{eqnarray*}
P ( x\in F | x\in F\cup G ) &=& P(x\mbox{ is open}) = p,
\\
P ( x\in G | x\in F \cup G ) &=& P(x\mbox{ is closed}) = 1-p.
\end{eqnarray*}
Summing
over $x\in\Lambda(n)$, we obtain
\begin{eqnarray*}
E \bigl(\llvert F\rrvert \bigr) &=& E \biggl(\sum_{x\in\Lambda(n)}1_{x\in F}
\biggr)
\\
&=& \sum_{x\in\Lambda(n)}P(x\in F) = \sum
_{x\in\Lambda(n)} P ( x\in F | x\in F\cup G ) P ( x\in F\cup G )
\\
&=& \sum_{x\in\Lambda(n)} p P ( x\in F\cup G ) = p E \bigl(
\llvert F\cup G\rrvert \bigr).
\end{eqnarray*}
Similarly, we have
\[
E \bigl(\llvert G\rrvert \bigr) = (1-p) E \bigl(\llvert F\cup G\rrvert \bigr).
\]
We wish to estimate the cardinality of $H$.
To this end, we write
\begin{eqnarray*}
\llvert H\rrvert &=& \biggl\llvert \bigcup_{C_1,C_2\in{\mathcal C}}
\bigl(\partial^{\mathrm{out}}C_1 \cap\partial^{\mathrm{out}}C_2
\cap\Lambda(n) \bigr) \biggr\rrvert
\\
&\leq&\sum_{C\in{\mathcal C}} \bigl\llvert \partial^{\mathrm{out}}C
\cap\Lambda(n) \bigr\rrvert - \biggl\llvert \bigcup_{C\in{\mathcal C}}
\partial^{\mathrm{out}}C\cap\Lambda(n) \biggr\rrvert
\\
&\leq&\sum_{C\in{\mathcal C}} \bigl\llvert \partial^{\mathrm{out}}C
\cap \Lambda(n) \bigr\rrvert - \llvert G 
\rrvert.
\end{eqnarray*}
Taking the expectation in this inequality, we obtain
\begin{eqnarray*}
E \bigl(\llvert H\rrvert \bigr) &\leq& E \biggl(\sum
_{C\in{\mathcal C}} \bigl\llvert \partial^{\mathrm{out}}C\cap\Lambda(n)
\bigr\rrvert \biggr)- E \bigl( \llvert G 
\rrvert \bigr)
\\
&=& E \biggl(\sum_{C\in{\mathcal C}} \bigl\llvert
\partial^{\mathrm{out}}C\cap \Lambda(n) \bigr\rrvert \biggr)- \frac{1-p}{p} E
\bigl( \llvert F 
\rrvert \bigr)
\\
&=& (1-p) E \biggl(\sum_{C\in{\mathcal C}} \biggl(
\frac{1}{1-p} \bigl\llvert \partial^{\mathrm{out}}C\cap\Lambda(n) \bigr
\rrvert - \frac{1}{p} \bigl\llvert C \cap\Lambda(n) \bigr\rrvert \biggr)
\biggr).
\end{eqnarray*}
For $A$ a subset of ${\mathbb Z}^d$, we define
\[
h(A) = \frac{1}{1-p} \bigl\llvert \{ x\in A\dvtx x\mbox{ is closed} \}
\bigr\rrvert - \frac{1}{p} \bigl\llvert \{ x\in A\dvtx x\mbox{ is open} \}
\bigr\rrvert.
\]
For $C$ an open cluster, we define
\[
\bC= C\cup\partial^{\mathrm{out}}C.
\]
With these definitions, we can rewrite the previous inequality as
\[
E \bigl(\llvert H\rrvert \bigr) \leq(1-p) E \biggl(\sum
_{C\in{\mathcal C}} h \bigl(\bC \cap\Lambda(n) \bigr) \biggr).
\]
Our\vspace*{1pt} next goal is to control the expectation on the right-hand side.
We first notice that,
for $x$ in the box $\Lambda(n)$,
the expected value of
$h(\bC(x) \cap\Lambda(n))$ is zero.

\begin{lemma}
For any $x\in\Lambda(n)$, we have
$E (h(\bC(x)
\cap\Lambda(n)
) )=0$.
\end{lemma}

\begin{pf}
Let $x\in\Lambda(n)$.
For any lattice animal $A$ containing $x$ and included in~$\Lambda(n)$,
we have
\[
P \bigl(C(x)\cap\Lambda(n)=A \bigr) = p^{\llvert A\rrvert }(1-p)^{\llvert \partial^{\mathrm{out}}A
\cap\Lambda(n)
\rrvert }.
\]
Summing over all such lattice animals $A$, we get
\[
1 = \sum_A 
p^{\llvert A\rrvert }(1-p)^{\llvert \partial^{\mathrm{out}}A
\cap\Lambda(n)
\rrvert }.
\]
Differentiating with respect to $p$, we obtain
\[
0 = \sum_A 
 \biggl(
\frac{\llvert A\rrvert }{p}- \frac{\llvert \partial^{\mathrm{out}}A
\cap\Lambda(n)
\rrvert }{1-p} \biggr) p^{\llvert A\rrvert }(1-p)^{\llvert \partial^{\mathrm{out}}A
\cap\Lambda(n)
\rrvert }
\]
and we notice that this last sum is equal to
$E (h(\bC(x)
\cap\Lambda(n)
) )$.
\end{pf}

It turns out that, for large clusters, the value
$h(\bC
\cap\Lambda(n)
)$ is close to $0$ with high probability. This is quantified by
the next proposition.
\section{The large deviation estimate}
The basic inequality leading to the control of
the two-arms event relies on the following large deviation estimate.
This estimate is a variant of the one stated in \cite{AKN,GGR}.
We have introduced an additional parameter $\ell$ and we use
Hoeffding's inequality.

\begin{proposition}\label{hcontrol}
For any $p$ in $]0,1[$, any $n\geq1$, $\ell\geq0$, we have
\begin{eqnarray*}
&& \forall x\in\Lambda(n+\ell), \forall k\geq1, \forall t \geq0
\\
&&\qquad P \bigl( \bigl\llvert h \bigl(\bC(x)\cap\Lambda(n) \bigr) \bigr\rrvert \geq t,
\bigl\llvert \bC(x)\cap\Lambda(n) \bigr\rrvert =k \bigr) \leq\exp
\biggl(-2p^2(1-p)^2 \frac{t^2}{k} \biggr)
\end{eqnarray*}
\end{proposition}

\begin{pf}
Let $x\in\Lambda(n+\ell)$.
In order to estimate the above probability, we build
$\bC(x)\cap\Lambda(n)$ in two steps.
First, we explore all the sites of
$\Lambda(n+\ell)\setminus\Lambda(n)$.
Second, we use a standard growth algorithm in
$\Lambda(n)$ to find
the sites belonging to
$\bC(x)\cap\Lambda(n)$. This algorithm is driven by a sequence
of i.i.d. Bernoulli random variables $(X_m)_{m\geq1}$
with parameter~$p$.
Let us describe precisely
this strategy.
The first step amounts to
condition on the percolation configuration in
$\Lambda(n+\ell)\setminus
\Lambda(n)$. We denote this configuration by
$\omega|_{\Lambda(n+\ell)\setminus
\Lambda(n)}$ and we write
\begin{eqnarray*}
&& P \bigl( \bigl\llvert h \bigl(\bC(x)\cap\Lambda(n) \bigr) \bigr\rrvert \geq
t, \bigl\llvert \bC(x)\cap \Lambda(n) \bigr\rrvert =k \bigr)
\\
&&\qquad = 
\sum_{\eta} P \bigl( \bigl
\llvert h \bigl(\bC(x)\cap\Lambda(n) \bigr) \bigr\rrvert \geq t, \bigl\llvert
\bC(x)\cap\Lambda(n) \bigr\rrvert =k, \omega |_{\Lambda(n+\ell)\setminus
\Lambda(n)}=\eta \bigr)
\\
&&\qquad = 
\sum_{\eta} P \bigl( \bigl
\llvert h \bigl(\bC(x)\cap\Lambda(n) \bigr) \bigr\rrvert \geq t, \bigl\llvert
\bC(x)\cap\Lambda(n) \bigr\rrvert =k \llvert \omega\rrvert _{\Lambda(n+\ell)\setminus
\Lambda(n)}=\eta
\bigr)
\\
&&\hspace*{46pt}{}\times P ( \omega|_{\Lambda(n+\ell)\setminus
\Lambda(n)}=\eta).
\end{eqnarray*}
The summation runs over all the percolation configurations $\eta$
in
$\Lambda(n+\ell)\setminus\break\Lambda(n)$.
Let us fix one such configuration~$\eta$.
The second step corresponds to the growth algorithm.
At each iteration, the algorithm updates three sets of sites:
\begin{itemize}
\item The set $A_k$: these are the active sites, which are to
be explored.

\item The set $O_k$: these are open sites, which belong to
$\bC(x)\cap\Lambda(n)$.

\item The set $C_k$: these are closed sites, which have been
visited by the algorithm.
\end{itemize}
All the sites of the sets $A_k$, $O_k$, $C_k$ are in $\Lambda(n)$.
Initially, we set $O_0=C_0=\varnothing$ and $A_0$ is the set of
the sites of
$\Lambda(n)$ which are connected to $x$ by an open path in~$\eta$.
Recall that a path is a sequence of sites such that each site is a
neighbor of its predecessor.
Thus
a site $y$ belongs to $A_0$ if and only if
\begin{eqnarray}
\exists z_0,\ldots,z_r\in\Lambda(n+\ell)\setminus
\Lambda(n)\qquad
z_0,\ldots,z_r
\mbox{ are open in }\eta,\nonumber
\\
\eqntext{z_0=x,\  z_0,\ldots,z_r,y
\mbox{ is a path}.}
\end{eqnarray}
Suppose that the sets $A_k,O_k,C_k$ are built and let us explain
how to build the sets
$A_{k+1},O_{k+1},C_{k+1}$.
If $A_k=\varnothing$, the algorithm terminates and
\[
\bC(x)\cap\Lambda(n) = O_k\cup C_k.
\]
If $A_k$ is not empty, we pick an element $x_k$ of $A_k$.
The site $x_k$ has not been explored previously, and its status will
be decided by the random variable $X_k$.
We consider two
cases, according to the value of $X_k$.
\begin{itemize}
\item $X_k=0$. The site $x_k$ is declared closed, and we set
\[
A_{k+1} = A_k\setminus\{ x_k \},\qquad
O_{k+1} = O_k,\qquad C_{k+1} = C_k
\cup\{ x_k \}.
\]

\item $X_k=1$.
The site $x_k$ is declared open, and we set
\begin{eqnarray*}
O_{k+1} &=& O_k\cup\{ x_k \}, \qquad
C_{k+1} = C_k,
\\
A_{k+1} &=& A_k\cup V_k\setminus \bigl(\{
x_k \}\cup O_k\cup C_k \bigr),
\end{eqnarray*}
\end{itemize}
where $V_k$ is the set of the sites of $\Lambda(n)$
which are neighbors of $x_k$ or
which are connected to $x_k$ by an open path
in
$\Lambda(n+\ell)\setminus\Lambda(n)$. More precisely,
a site $y$ of $\Lambda(n)$ belongs to $V_k$ if and only if
it is a neighbor of $x_k$ or
\begin{eqnarray}
\exists z_1,\ldots,z_r\in\Lambda(n+\ell)\setminus
\Lambda(n)\qquad
z_1,\ldots,z_r\mbox{ are open
in }\eta,\nonumber
\\
\eqntext{x_k, z_1,\ldots,z_r,y\mbox{ is a path}.}
\end{eqnarray}
Since $O_k\cup C_k\cup A_k$ is included in $\Lambda(n)$ and the
sequence of sets
$O_k\cup C_k, k\geq0$, is increasing,\vspace*{1pt} necessarily $A_k$ is empty after at
most $|\Lambda(n)|$ steps
and the algorithm terminates.
Suppose
$ |\bC(x)\cap\Lambda(n) |=k$.
This means that the growth algorithm stops after having
explored $k$ sites in $\Lambda(n)$. The status of these $k$ sites is given
by the first $k$ variables of the sequence
$(X_m)_{m\geq1}$, so that
\begin{eqnarray*}
\bigl\llvert C(x)\cap\Lambda(n) \bigr\rrvert &=& X_1+
\cdots+X_k,
\\
\bigl\llvert \partial^{\mathrm{out}}C(x)\cap\Lambda(n) \bigr\rrvert &=& k-
(X_1+\cdots +X_k )
\end{eqnarray*}
and
\begin{eqnarray*}
h \bigl(\bC(x) \cap\Lambda(n) \bigr) &=& \frac{1}{1-p} \bigl\llvert
\partial^{\mathrm{out}}C(x)\cap\Lambda(n) \bigr\rrvert - \frac{1}{p} \bigl
\llvert C(x)\cap \Lambda(n) \bigr\rrvert
\\
&=& \frac{1}{1-p} \bigl(k-(X_1+ \cdots+X_k) \bigr)
- \frac{1}{p} (X_1+ \cdots+X_k )
\\
&=& \frac{pk-
(X_1+\cdots+X_k)}{p(1-p)}.
\end{eqnarray*}
Therefore, we can write
\begin{eqnarray*}
&& P \bigl( \bigl\llvert h \bigl(\bC(x)\cap\Lambda(n) \bigr) \bigr\rrvert \geq
t, \bigl\llvert \bC(x)\cap \Lambda(n) \bigr\rrvert =k \llvert \omega\rrvert
_{\Lambda(n+\ell)\setminus
\Lambda(n)}=\eta \bigr)
\\
&&\qquad = P \biggl( \biggl\llvert \frac{pk-
(X_1+\cdots+X_k)}{p(1-p)} \biggr\rrvert \geq t,
\bigl\llvert \bC(x)\cap\Lambda(n) \bigr\rrvert =k \llvert \omega\rrvert
_{\Lambda(n+\ell)\setminus
\Lambda(n)}=\eta \biggr)
\\
&&\qquad \leq P \biggl( \biggl\llvert \frac{pk-
(X_1+\cdots+X_k)}{p(1-p)} \biggr\rrvert \geq t
\llvert \omega\rrvert _{\Lambda(n+\ell)\setminus
\Lambda(n)}=\eta \biggr)
\\
&&\qquad = P \bigl( \llvert X_1+\cdots+X_k-pk \rrvert
\geq t {p(1-p)} \bigr)
\\
&&\qquad \leq2\exp \biggl(-\frac{2}{k}t^2p^2(1-p)^2
\biggr).
\end{eqnarray*}
For the last step, we have applied Hoeffding's inequality
\cite{HO} (one could also use the earlier inequality due to Bernstein
\cite{BE}).
The above inequality is uniform with respect to the
configuration~$\eta$. Plugging this bound in the initial summation,
we obtain the desired estimate.
\end{pf}

\section{The central inequality}
We will now put together the previous estimates in order to obtain an
inequality between the probability of the two-arms event and the
number of clusters in the collection~${\mathcal C}$.
Our goal is to bound the expectation
\[
E \biggl(\sum_{C\in{\mathcal C}} h \bigl(\bC\cap\Lambda(n) \bigr)
\biggr).
\]
Let ${\mathcal E}$ be the event
\[
{\mathcal E}= \bigl\{ \forall C\in{\mathcal C},\bigl\llvert h \bigl(\bC\cap
\Lambda(n) \bigr) \bigr\rrvert < (\ln n) \bigl\llvert \bC\cap\Lambda(n) \bigr
\rrvert ^{1/2} \bigr\}.
\]
On the event ${\mathcal E}$, we bound the sum as follows:
\begin{eqnarray*}
\sum_{C\in{\mathcal C}} \bigl\llvert h \bigl(\bC\cap\Lambda(n)
\bigr) \bigr\rrvert & \leq&\sum_{C\in{\mathcal C}} (\ln n) \bigl
\llvert \bC\cap\Lambda(n) \bigr\rrvert ^{1/2}
\\
& \leq&(\ln n) \sqrt{ \llvert {\mathcal C}\rrvert } \biggl( \sum
_{C\in{\mathcal C}} \bigl\llvert \bC\cap\Lambda(n) \bigr\rrvert
\biggr)^{1/2}.
\end{eqnarray*}
A site $x$ belongs to at most $2d$ sets of the collection
$ \{ \bC\dvtx C\in{\mathcal C} \}$, therefore,
\[
\sum_{C\in{\mathcal C}} \bigl\llvert \bC\cap\Lambda(n) \bigr
\rrvert \leq2d \bigl\llvert \Lambda(n) \bigr\rrvert .
\]
If ${\mathcal E}$ does not occur, then we use the inequality
\[
\forall C\in{\mathcal C}\qquad\bigl\llvert h \bigl(\bC\cap\Lambda(n) \bigr)
\bigr\rrvert \leq \frac{1}{p(1-p)} \bigl\llvert \bC\cap\Lambda(n) \bigr\rrvert
\]
and we bound the sum as follows:
\[
\sum_{C\in{\mathcal C}} \bigl\llvert h \bigl(\bC\cap\Lambda(n)
\bigr) \bigr\rrvert \leq \frac{1}{p(1-p)} \sum_{C\in{\mathcal C}}
\bigl\llvert \bC\cap\Lambda(n) \bigr\rrvert \leq\frac{2d}{p(1-p)} \bigl\llvert
\Lambda(n) \bigr\rrvert. %
\]
%
We bound the probability of the complement of ${\mathcal E}$ with the
help of
Proposition~\ref{hcontrol}:
\begin{eqnarray*}
P \bigl({\mathcal E}^c \bigr)& =& P \bigl( \exists C\in{\mathcal C},
\bigl\llvert h \bigl(\bC\cap\Lambda(n) \bigr) \bigr\rrvert \geq(\ln n) \bigl
\llvert \bC\cap\Lambda(n) \bigr\rrvert ^{1/2} \bigr)
\\
& \leq& P \bigl( \exists x\in\Lambda(n),
\bigl\llvert h \bigl(
\bC(x)\cap\Lambda(n) \bigr) \bigr\rrvert \geq(\ln n) \bigl\llvert \bC(x) \cap
\Lambda(n) \bigr\rrvert ^{1/2} \bigr) 
\\
& \leq& \sum_{
x\in\Lambda(n)} 
\sum
_{k=1}^{\llvert \Lambda(n)\rrvert } P \bigl( \bigl\llvert \bC(x) \cap
\Lambda(n) \bigr\rrvert =k,
\\
&&\hspace*{65pt}
\bigl\llvert h \bigl(\bC(x)\cap \Lambda(n) \bigr) \bigr\rrvert \geq(\ln n)
\bigl\llvert \bC(x)\cap\Lambda(n) \bigr\rrvert ^{1/2} \bigr)
\\
& \leq&\bigl\llvert \Lambda(n) \bigr\rrvert ^2 
2\exp
\bigl(-{2} (\ln n)^2 p^2(1-p)^2 \bigr).
\end{eqnarray*}
Putting together the previous inequalities, we obtain
\begin{eqnarray*}
E \bigl(\llvert H\rrvert \bigr) &\leq&2d(\ln n)\sqrt{ \bigl\llvert \Lambda(n)
\bigr\rrvert } E \bigl( \sqrt{\llvert {\mathcal C}\rrvert } \bigr)
\\
&&{}+ \frac{4d}{p(1-p)} \bigl\llvert \Lambda(n) \bigr\rrvert ^3
\exp \bigl(-{2} (\ln n)^2 p^2(1-p)^2
\bigr).
\end{eqnarray*}
%

\begin{definition}
For $x\in{\mathbb Z}^d$ and $n\geq1$, we define
the event $\ta(x,n)$ by:
%
\[
\ta(x,n) = \left\{
\matrix{
\mbox{in the configuration restricted to $x+\Lambda(n)$}\vspace*{2pt}\cr
\mbox{two neighbors of $x$ are connected to the boundary}\vspace*{2pt}\cr
\mbox{of the box $x+\Lambda(n)$ by two disjoint open clusters}}\right\}.
\]
\end{definition}

If $x$ belongs to $\Lambda(n)$ and the event $\ta(x,2n+\ell)$ occurs,
then $x$ belongs to $H$ as well. Thus,
\[
\llvert H\rrvert \geq\sum_{x\in\Lambda(n)}1_{\mathrm{two}\mbox{-}\mathrm{arms}(x,2n+\ell)}
\]
and taking expectation, we obtain the following central inequality.

\begin{lemma}
\label{ci}
For any $p$ in $]0,1[$, any $n\geq1$, $\ell\geq0$, we have
the inequality
\label{ce}
\begin{eqnarray*}
&& P \bigl(\ta(0,2n+\ell) \bigr)
\\
&&\qquad \leq\frac{2d\ln n}{\sqrt{ \llvert \Lambda(n) \rrvert }} E \bigl( \sqrt{\llvert {
\mathcal C}\rrvert } \bigr)
+ \frac{4d}{p(1-p)} \bigl\llvert \Lambda(n) \bigr\rrvert ^2
\exp \bigl(-{2} ( \ln n)^2 p^2(1-p)^2
\bigr).
\end{eqnarray*}
\end{lemma}

In order to obtain the initial estimate on the two-arms event stated
in the \hyperref[sec1]{Introduction},
we remark that
the cardinality of ${\mathcal C}$ is bounded by the cardinality of
$\partial^{\mathrm{in}}\Lambda(n)$,
because different clusters of ${\mathcal C}$ intersect
$\partial^{\mathrm{in}}\Lambda(n)$ at different sites.
Taking $\ell=0$ in the inequality, we have
\begin{eqnarray*}
&& P \bigl(\ta(0,2n) \bigr)
\\
&&\qquad \leq 2d(\ln n) \biggl( \frac{ \llvert \partial^{\mathrm{in}}\Lambda(n) \rrvert }{
\llvert \Lambda(n) \rrvert }
\biggr)^{1/2}
+ \frac{4d}{p(1-p)} \bigl\llvert \Lambda(n) \bigr\rrvert ^2
\exp \bigl(-{2} (\ln n)^2 p^2(1-p)^2
\bigr).
\end{eqnarray*}
This inequality readily implies the initial estimate stated in the \hyperref[sec1]{Introduction}.

\begin{proposition}\label{ie}
Let $d\geq2$ and let $p\in\,]0,1[$.
There exists a constant $\kappa$ depending on $d$ and $p$ only such that
\[
\forall n\geq1\qquad P_{p} \bigl(\ta(0,n) \bigr) \leq
\frac{\kappa\ln n}{\sqrt{n}}.
\]
\end{proposition}

In order to improve this estimate on the two-arms exponent, we will
try to improve the estimate on the cardinality of ${\mathcal C}$.
\section{Lower bound for the connection probability}
For $x,y$ two sites belonging to a box $\Lambda$, we define the event
\begin{eqnarray*}
&& \{ x\longleftrightarrow y\mbox{ in }\Lambda\}
\\
&&\qquad = \{\mbox{the sites
$x$ and $y$ are joined by an open path of sites inside $\Lambda$}\}. %
\end{eqnarray*}
%
The next lemma gives a polynomial lower bound for the probability of connection
of two sites of $\Lambda(n)$ if one allows the path to be in $\Lambda(2n)$.
At criticality, the expected behavior is indeed a power of $n$,
but with a different exponent. In Lemma~1.1 of \cite{NK},
Kozma and Nachmias
derive a smaller lower bound, however, only paths staying inside
$\Lambda(n)$
are allowed.

\begin{lemma}
\label{connect}
There exists a positive constant $c$ which depends only on the dimension~$d$
such that,
for $n\geq1$,
\[
\forall x,y\in\Lambda(n)\qquad P_{p_c} \bigl(x\longleftrightarrow y
\mbox{ in }\Lambda(2n) \bigr) \geq\frac{c}{n^{2(d-1)d}}.
\]
\end{lemma}

\begin{pf}
The basic ingredient to prove Lemma~\ref{connect} is the following
lower bound.
For any box~$\Lambda$ centered at $0$, we have
\[
\sum_{x\in\partial^{\mathrm{in}}\Lambda} P_{p_c} (0\longleftrightarrow x
\mbox{ in }\Lambda) \geq1.
\]
%
This lower bound is proved in
Lemma~3.1 of \cite{NK},
or in the proof of Theorem~5.3 of \cite{GRB}.
The reason is that, by an argument
due to Hammersley \cite{HA},
if the converse inequality holds,
then this implies that the probability of long
connections decays exponentially fast with the distance, and the system
would be in the subcritical regime.
Applying\vspace*{1pt} the above inequality to the box $\Lambda(n)$, we conclude
that there
exists $x^*$ in
$\partial^{\mathrm{in}}\Lambda(n)$ such that
\[
P_{p_c} \bigl(0\longleftrightarrow x^*\mbox{ in }\Lambda(n) \bigr) \geq
\frac{1}{\llvert \partial^{\mathrm{in}}\Lambda(n)\rrvert } \geq\frac{1}{(2d)(2n+1)^{d-1}}.
\]
Without loss of generality,
we can suppose that $x^*$ belongs to $\{ n \}\times{\mathbb Z}^{d-1}$.
Let us set $e_1=(1,0,\ldots,0)$.
By the FKG inequality and the symmetry of the model, we have
\begin{eqnarray*}
&& P_{p_c} \bigl(0\longleftrightarrow2ne_1\mbox{ in }
\Lambda(n)\cup \bigl( 2ne_1 +\Lambda(n) \bigr) \bigr)
\\
&&\qquad \geq P_{p_c} \bigl(0\longleftrightarrow x^*\mbox{ in }
\Lambda(n)\cup \bigl( 2ne_1 +\Lambda(n) \bigr),
\\
&&\hspace*{53pt} x^*
\longleftrightarrow2ne_1 \mbox{ in } \Lambda(n)\cup \bigl(
2ne_1 +\Lambda(n) \bigr) \bigr)
\\
&&\qquad \geq P_{p_c} \bigl(0\longleftrightarrow x^*\mbox{ in }
\Lambda(n)\cup \bigl( 2ne_1 +\Lambda(n) \bigr) \bigr)
\\
&&\quad\qquad{}\times P_{p_c} \bigl( x^*\longleftrightarrow2ne_1
\mbox{ in } \Lambda(n)\cup \bigl( 2ne_1 +\Lambda(n) \bigr) \bigr)
\\
&&\qquad \geq P_{p_c} \bigl(0\longleftrightarrow x^*\mbox{ in }
\Lambda(n) \bigr) P_{p_c} \bigl(x^*\longleftrightarrow2ne_1
\mbox{ in } 2ne_1+\Lambda(n) \bigr)
\\
&&\qquad \geq \biggl( \frac{1}{(2d)(2n+1)^{d-1}} \biggr)^2.
\end{eqnarray*}
By symmetry,
the same inequality holds for the other axis directions.
Let now $x,y$ be two sites in $\Lambda(n)$ with coordinates
\[
x=(x_1,\ldots,x_d),\qquad y=(y_1,
\ldots,y_d).
\]
We suppose first that $y_i-x_i$ is even, for $1\leq i\leq d$,
and we set
\[
z_0=x, z_1=(y_1,x_2,
\ldots,x_d),\ldots, z_{d-1}=(y_1,
\ldots,y_{d-1},x_d), z_{d}=y.
\]
Again by the FKG inequality, we have
\begin{eqnarray*}
&& P_{p_c} \bigl(x\longleftrightarrow y\mbox{ in }\Lambda(2n) \bigr)
\\
&&\qquad \geq P_{p_c} \bigl( \forall i\in\{ 0,\ldots,d-1 \}, z_i
\longleftrightarrow z_{i+1}\mbox{ in }\Lambda(2n) \bigr)
\\
&&\qquad \geq\prod_{0\leq i\leq d-1} P_{p_c} \bigl(
z_i \longleftrightarrow z_{i+1}\mbox{ in }\Lambda(2n)
\bigr).
\end{eqnarray*}
Let $i\in\{ 0,\ldots,d-1 \}$ and let $n_i=(y_i-x_i)/2$.
We have $n_i\leq n$ and
\[
\bigl(z_i+\Lambda(n_i) \bigr) \cup \bigl(
z_{i+1} +\Lambda(n_i) \bigr) 
\subset\Lambda(2n), %
\]
whence
\begin{eqnarray*}
&&P_{p_c} \bigl( z_i\longleftrightarrow z_{i+1}
\mbox{ in }\Lambda(2n) \bigr)
\\
&&\qquad \geq P_{p_c} \bigl( z_i
\longleftrightarrow z_{i+1}\mbox{ in } \bigl(z_i+
\Lambda(n_i) \bigr) \cup \bigl( z_{i+1} +
\Lambda(n_i) \bigr) \bigr)
\\
&&\qquad \geq \biggl( \frac{1}{(2d)(2n_i+1)^{d-1}} \biggr)^2.
\end{eqnarray*}
Coming back to the previous inequality, we obtain
%
\[
P_{p_c} \bigl(x\longleftrightarrow y\mbox{ in }\Lambda(2n) \bigr) \geq
\prod_{0\leq i\leq d-1} \biggl( \frac{1}{(2d)(2n_i+1)^{d-1}}
\biggr)^2 \geq\frac{c}{n^{2(d-1)d}}, %
\]
%
where the last inequality holds for some positive constant $c$.
In the general case, if $x\neq y$ and if $x_i-y_i$ is not even
for some $1\leq i\leq d$,
we can find $z$ in $\Lambda(n)$ such that
$\llvert z-x\rrvert \leq\llvert y-x\rrvert $ and
%
\[
\forall i\in\{ 1,\ldots,d \}\qquad\llvert z_i-y_i
\rrvert \leq1, \qquad z_i-x_i\mbox{ is even}.
\]
We then use the FKG inequality to write
\[
P_{p_c} \bigl(x\longleftrightarrow y\mbox{ in }\Lambda(2n) \bigr) \geq
P_{p_c} \bigl(x\longleftrightarrow z\mbox{ in }\Lambda(2n) \bigr)
P_{p_c} \bigl(z\longleftrightarrow y\mbox{ in }\Lambda(2n) \bigr).
\]
The probability of connection between $x$ and $z$ is controlled with the
help of the previous case, while the
probability of connection between $z$ and $y$ is larger than~$(p_c)^d$.
\end{pf}

\section{Two-arms for distant sites}
We derive here an estimate for the two-arms event associated to two
distant sites, which we define next.

\begin{definition}
For $n,\ell\geq1$ and two sites $a,b$ belonging to $\Lambda(n)$,
we define the event $\ta(\Lambda(n),a,b,\ell)$ as follows:
%
\[
\ta \bigl(\Lambda(n),a,b,\ell \bigr) =
\left\{
\matrix{
\mbox{the open clusters of $a$ and $b$ in $\Lambda(n+\ell)$}\vspace*{2pt}\cr
\mbox{are disjoint and they intersect $\partial^{\mathrm{in}}\Lambda(n+\ell)$}}
\right\}. %
\]
\end{definition}

We will establish an inequality linking the two-arms event for distant sites
to
the two-arms event for neighboring sites.

%
\begin{figure}

\includegraphics{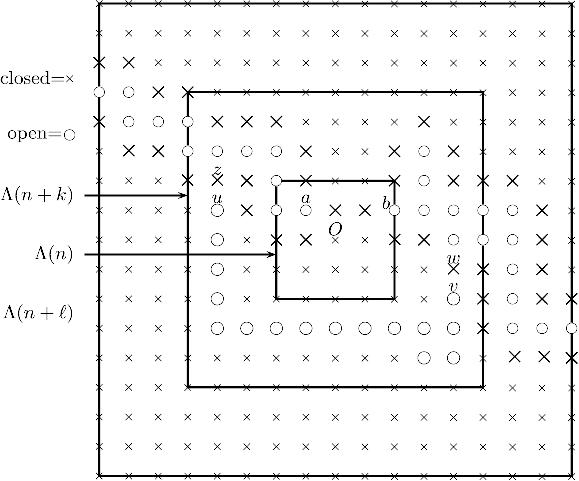}

\label{fig3}
\end{figure}

%
\begin{lemma}\label{ineqab} Let $p\in\,]0,1[$.
For any $n,\ell\geq1$ and any $a,b\in\Lambda(n)$, we have
\[
\forall k\leq\ell\qquad P \bigl( \ta \bigl(\Lambda(n),a,b,\ell \bigr) \bigr)
\leq\frac{
3^{4d}
}{p} (n+k)^{2d} \frac{
P (\ta(0,\ell-k) )}{
P ( a \longleftrightarrow b\mbox{ in }\Lambda(n+k) )
}. %
\]
\end{lemma}

\begin{pf}
Let $n,\ell\geq1$, let $k\leq\ell$ and let $a,b\in\Lambda(n)$.
We denote by $C(a)$ and $C(b)$ the
open clusters of $a$ and $b$ in $\Lambda(n+\ell)$.
We write
\[
P \bigl( \ta \bigl(\Lambda(n),a,b,\ell \bigr) \bigr) = \sum
_{A,B} P \bigl( C(a)=A, C(b)=B \bigr), %
\]
where the sum runs over the pairs $A,B$ of connected subsets
of $\Lambda(n+\ell)$ such that
\begin{eqnarray*}
A\cap B &=& \varnothing,\qquad a\in A,
\\
A\cap\partial^{\mathrm{in}}\Lambda(n+ \ell) &\neq& \varnothing,\qquad b\in B,
\\
B\cap\partial^{\mathrm{in}}\Lambda(n+\ell) &\neq&\varnothing. %
\end{eqnarray*}
For $E$ a finite subset of ${\mathbb Z}^d$, we define
\[
\bE= E\cup\partial^{\mathrm{out}}E,\qquad\Delta E = \partial^{\mathrm{in}}
\bigl( (\bE)^c \bigr).
\]
Equivalently, we have
\[
\Delta E = \bigl\{ z\notin E\cup\partial^{\mathrm{out}}E\dvtx z\mbox{ is the
neighbor of a point in } \partial^{\mathrm{out}}E \bigr\}.
\]
Let $a,b\in\Lambda(n)$ and let $A,B$
be two connected subsets
of $\Lambda(n+\ell)$ as above.
Suppose that the open clusters of $a$ and $b$ in $\Lambda(n+\ell)$
are exactly $A$ and $B$, that is, we have
$C(a)=A$ and $C(b)=B$.
Suppose that
$\partial^{\mathrm{out}}A\cap
\partial^{\mathrm{out}}B\cap\Lambda(n)\neq\varnothing$.
Then the event
$\ta(z,\ell)$ occurs, where $z$ is any point in the previous intersection.
Suppose next that
\[
\partial^{\mathrm{out}}A\cap\partial^{\mathrm{out}}B\cap\Lambda(n)=\varnothing.
\]
We will transform the configuration in $\Lambda(n)$ in order
to create a two-arms event.
The idea is that, for $k\leq\ell$, the sets
$\Delta A$ and $\Delta B$ are rather likely to be connected by an open path
inside
$\Lambda(n+k)\setminus(\bA\cup\bB)$.
By modifying the status of one site in $\partial^{\mathrm{out}}B$,
we can then create a connection between $\Delta A$ and
$\partial^{\mathrm{in}}\Lambda(n+\ell)$, which does not use the sites of $A$.
Let us make this strategy more precise.
Any open path joining $a$ to $b$ in $\Lambda(n+k)$ has to go through
both $\Delta A$ and $\Delta B$, thus
%
\begin{eqnarray*}
&& P \bigl( a \longleftrightarrow b\mbox{ in }\Lambda(n+k) \bigr)
\\
&&\qquad \leq
P \bigl( \Delta A\cap\Lambda(n+k) \longleftrightarrow\Delta B\cap\Lambda(n+k)
\mbox{ in } \Lambda(n+k)\setminus(\bA\cup\bB) \bigr). 
\end{eqnarray*}
The event
$ \{ C(a)=A, C(b)=B \}$
depends only on the sites in $\bA\cup\bB$, hence it is independent
from the
event above, therefore,
\begin{eqnarray*}
&& P \bigl( C(a)=A, C(b)=B,
\\
&&\hspace*{13pt} \Delta A\cap\Lambda(n+k) \longleftrightarrow\Delta B
\cap\Lambda(n+k) \mbox{ in }\Lambda (n+k)\setminus(\bA \cup\bB) \bigr)
\\
&&\qquad \geq P \bigl( C(a)=A, C(b)=B \bigr)\times
P \bigl( a
\longleftrightarrow b\mbox{ in }\Lambda(n+k) \bigr).
\end{eqnarray*}
Let ${\mathcal E}$ be the event
\[
{\mathcal E}= \bigl\{ C(a)\cap\partial^{\mathrm{in}}\Lambda(n+\ell) \neq
\varnothing, C(b)\cap\partial^{\mathrm{in}}\Lambda(n+\ell) \neq\varnothing, C(a)
\cap C(b)=\varnothing \bigr\}. %
\]
Summing the previous inequality over $A,B$, we get
%
\begin{eqnarray*}
&& \biggl(\sum_{A,B} P \bigl( C(a)=A, C(b)=B \bigr)
\biggr) P \bigl( a \longleftrightarrow b\mbox{ in }\Lambda(n+k) \bigr)
\\
&&\qquad \leq \sum_{A,B} P
\pmatrix{
C(a)=A, C(b)=B\vspace*{2pt}\cr
\Delta A\cap\Lambda(n+k)\longleftrightarrow\Delta B \cap\Lambda(n+k) \mbox{ in }\Lambda (n+k)\setminus(\bA \cup\bB)}
\\
&&\qquad \leq P
\pmatrix{
 {\mathcal E}, \Delta C(a)\cap\Lambda(n+k)
\longleftrightarrow \Delta C(b)\cap\Lambda(n+k)
\vspace*{2pt}\cr
\mbox{in }\Lambda (n+k)\setminus
\bigl( \overline{C(a)}\cup\overline{C(b)} \bigr) } 
\\
&&\qquad \leq P
\pmatrix{
{\mathcal E}, \exists u\in\Delta C(a)\cap\Lambda(n+k), \exists v\in\Delta C(b)\cap \Lambda(n+k) \vspace*{2pt}\cr
u\longleftrightarrow v \mbox { in } \Lambda(n+k)\setminus \bigl(\overline{C(a)}\cup\overline{C(b)} \bigr)}
\\
&&\qquad \leq\sum_{u,v\in\Lambda(n+k)}  P
\pmatrix{
{\mathcal E}, u\in \Delta C(a), v\in\Delta C(b) \vspace*{2pt}\cr
u \longleftrightarrow v \mbox{ in }\Lambda(n+k) \setminus \bigl(\overline{C(a)}\cup
\overline{C(b)} \bigr)}.
\end{eqnarray*}
Let us consider the event inside the probability appearing in this
sum.
Let $z$ (resp.,~$w$)
be a neighbor of $u$ (resp., $v$) belonging to
$\partial^{\mathrm{out}}C(a)$
[resp., $\partial^{\mathrm{out}}C(b)$].
Suppose that we change the status of $w$ to open.
The site $u$ is connected to $v$ by an open path, and $v$ is now
connected to $w$ and $C(b)$, hence to $\partial^{\mathrm{in}}\Lambda(n+\ell
)$, and this
connection does not use any site of $C(a)$. Thus, the site $z$, which is
closed, will admit two neighbors which are connected to
$\partial^{\mathrm{in}}\Lambda(n+\ell)$:  the site $u$ and another one
belonging to $C(a)$,
and these two neighbors do not belong to the same cluster
in $\Lambda(n+\ell)$.
Therefore, the event
$\ta(z,\ell-k)$ occurs, and we conclude that
\begin{eqnarray*}
&& P
\pmatrix{
C(a)\cap\partial^{\mathrm{in}}\Lambda(n+\ell) \neq \varnothing, C(b)\cap\partial^{\mathrm{in}}\Lambda(n+\ell) \neq\varnothing, C(a)\cap C(b)=\varnothing\vspace*{2pt}\cr
u\in\Delta C(a), v\in\Delta C(b), u \longleftrightarrow v \mbox{ in }\Lambda(n+k)\setminus \bigl(\overline {C(a)}\cup \overline{C(b)} \bigr)
}
\\
&&\qquad\leq\frac{4d^2}{p} P \bigl(\ta(0,\ell-k) \bigr).
\end{eqnarray*}
Plugging this inequality in the previous sum, we obtain
\begin{eqnarray*}
&& P \bigl( \ta \bigl(\Lambda(n),a,b,\ell \bigr) \bigr)
\\
&&\qquad  \leq
\sum_{u,v\in\Lambda(n+k)} \frac{4d^2}{p} \frac{
P (\ta(0,\ell-k) )}{
P ( a \longleftrightarrow b\mbox{ in }\Lambda(n+k) )
}
\\
&&\qquad \leq 
\bigl\llvert \Lambda(n+k) \bigr\rrvert ^2
\frac{4d^2}{p} \frac{
P (\ta(0,\ell-k) )}{
P ( a \longleftrightarrow b\mbox{ in }\Lambda(n+k) )
}
\\
&&\qquad  \leq\frac{
3^{4d}
}{p} (n+k)^{2d} \frac{
P (\ta(0,\ell-k) )}{
P ( a \longleftrightarrow b\mbox{ in }\Lambda(n+k) )
}.
\end{eqnarray*}
This is the inequality we wanted to prove.
\end{pf}

We derive next an estimate for the two-arms event associated to a box.
For $n,\ell\geq1$, we define
the event $\ta(\Lambda(n),\ell)$ as follows:
%
\[
\ta \bigl(\Lambda(n),\ell \bigr) =
\left\{
\matrix{
\mbox{there exist two distinct open clusters}\vspace*{2pt}\cr
\mbox{in $\Lambda(n+\ell)$ joining $ \Lambda(n)$ to $\partial^{\mathrm{in}}\Lambda(n+\ell)$}
}
\right\}.
\]
%

\begin{corollary}\label{intwoarms}
For any $n\geq1$, $\ell\geq n$, we have
\[
P \bigl( \ta \bigl(\Lambda(n),\ell \bigr) \bigr) \leq
\frac{
3^{9d}
}{p} \frac{
n^{4d-2}
P (\ta(0,\ell-n) )}{
\inf\{
P ( a \longleftrightarrow b\mbox{ in }\Lambda(2n) )\dvtx
a,b\in
\partial^{\mathrm{in}}\Lambda(n)
\}
}. %
\]
\end{corollary}

\begin{pf}
From the definition of the two-arms event, we have
\[
\ta \bigl(\Lambda(n),\ell \bigr) = \bigcup_{a,b\in\partial^{\mathrm{in}}\Lambda
(n)} \ta
\bigl(\Lambda(n),a,b,\ell \bigr).
\]
Therefore, applying the inequality of Lemma~\ref{ineqab} with $k=n$, we
obtain
\begin{eqnarray*}
&& P \bigl( \ta \bigl(\Lambda(n),\ell \bigr) \bigr)
\\
&&\qquad \leq\sum
_{a,b\in\partial^{\mathrm{in}}\Lambda(n)} P \bigl(\ta \bigl(\Lambda(n),a,b,\ell \bigr) \bigr)
\\
&&\qquad  \leq\sum_{a,b\in\partial^{\mathrm{in}}\Lambda(n)} \frac{
3^{4d}
}{p}
\frac{
(2n)^{2d}
P (\ta(0,\ell-n) )}{
P ( a \longleftrightarrow b\mbox{ in }\Lambda(2n) )
}
\\
&&\qquad  \leq\frac{
3^{4d}
}{p} \frac{
4d^2(2n+1)^{2d-2}
(2n)^{2d}
P (\ta(0,\ell-n) )}{
\inf\{
P ( a \longleftrightarrow b\mbox{ in }\Lambda(2n) )\dvtx
a,b\in
\partial^{\mathrm{in}}\Lambda(n)
\}
}.
\end{eqnarray*}
%
This yields the desired inequality.
\end{pf}

%
\begin{corollary}
\label{ftwoarms}
We have
\[
\lim_{n\to\infty} P \bigl(\ta \bigl(\Lambda(n),n^{4d^2+4d-3}
\bigr) \bigr) = 0.
\]
\end{corollary}

\begin{pf}
We apply the inequality given in Corollary~\ref{intwoarms}.
We use Proposition~\ref{ie} to control the probability of the two-arms event
and Lemma~\ref{connect} to control from below the connection probability.
We obtain
\[
P \bigl( \ta \bigl(\Lambda(n),\ell \bigr) \bigr) 
\leq\frac{
3^{9d}
}{pc} {n^{2d^2+2d-2}} \frac{\kappa\ln( \ell-n)}{\sqrt{\ell-n}}. %
\]
%
We take $\ell= n^{4d^2+4d-3}$ in this
inequality and we send $n$ to $\infty$.
\end{pf}

For $d=3$, this yields the exponent
$4d^2+4d-3=45$.
\section{Control on the number of arms}
We try next to improve the previous estimates.
The idea is the following. With the help of
Corollary~\ref{ftwoarms}, we will improve slightly the control on the
number of clusters in the collection ${\mathcal C}$ [these are the clusters
intersecting both
$\Lambda(n)$ and $\partial^{\mathrm{in}}\Lambda(n+\ell)$].
Thanks to the central inequality stated in Lemma~\ref{ci},
this will permit to improve the
bound on the two-arms event for a site, and subsequently the
bound on the two-arms event for a box.
This leads to a better exponent in Corollary~\ref{ftwoarms}.
We can then iterate this scheme to improve
further the exponents.
Unfortunately, the sequence of exponents converges geometrically and the
final result is still quite weak.

Let $n,\ell,k$ be three integers, with $k\leq n\leq\ell$.
Let $\Lambda_i$, $i\in I$, be a collection of boxes which are
translates of
$\Lambda(k)=[-k,k]^d$, which are included in
$\Lambda(n)$
and which covers the inner boundary
$\partial^{\mathrm{in}}\Lambda(n)$.
Such a covering can be realized with disjoint boxes if $2n+1$ is a multiple
of $2k+1$, otherwise we do not require that the boxes
are disjoint.
In any case, there exists such a covering
$\Lambda_i$, $i\in I$,
whose cardinality $\llvert I\rrvert $
satisfies
\[
\llvert I \rrvert \leq2d \biggl(2\frac{n}{k} \biggr)^{d-1}.
\]
Let us fix such a covering.
Given
a percolation configuration in
$\Lambda(n+\ell)$,
a box $\Lambda_i$ of the covering is said to be good if the event
$\ta(\Lambda_i,\ell)$ does not occur.
Let us compute the expected number of bad boxes:
\begin{eqnarray*}
E
\pmatrix{
\mbox{number of bad boxes in}\vspace*{2pt}\cr
\mbox{the collection $\Lambda_i$, $i\in I$}
}
& =&  E \biggl( \sum
_{i\in I} 1_{\mathrm{the\ box}\ \Lambda_i\ \mathrm{is\ bad}} \biggr)
\\
& =& \llvert I \rrvert P \bigl(\ta \bigl( \Lambda(k),\ell \bigr) \bigr).
\end{eqnarray*}
The clusters of the collection~${\mathcal C}$ intersect $\partial^{\mathrm{in}}\Lambda(n)$, hence
they have to go into one box of
the collection
$\Lambda_i$, $i\in I$.
If two clusters of ${\mathcal C}$ intersect the same box $\Lambda_i$,
this box has to
be bad, because these two clusters go all the way until
$\partial^{\mathrm{in}}\Lambda(n+\ell)$, hence they realize the event
$\ta(\Lambda_i,\ell)$.
Thus, a good box of
the collection
$\Lambda_i$, $i\in I$, meets at most one cluster of ${\mathcal C}$.
Moreover, a bad box
of the collection
$\Lambda_i$, $i\in I$, meets at most $ \llvert \partial^{\mathrm{in}}\Lambda
(k) \rrvert $
clusters of ${\mathcal C}$.
We conclude that
\[
\llvert {\mathcal C}\rrvert \leq
\pmatrix{
\mbox{number of good boxes in}\vspace*{2pt}\cr
\mbox{the collection $\Lambda_i$, $i\in I$}
}
+ \bigl\llvert
\partial^{\mathrm{in}} \Lambda(k) \bigr\rrvert \times
\pmatrix{
\mbox{number of bad boxes in}\vspace*{2pt}\cr
\mbox{the collection $\Lambda_i$, $i\in I$}}.
\]
We bound the number of good boxes by $\llvert I\rrvert $ and we take
the expectation in this inequality. We obtain
\begin{eqnarray*}
E \bigl( \llvert {\mathcal C}\rrvert \bigr) & \leq&\llvert I \rrvert + \bigl
\llvert \partial^{\mathrm{in}}\Lambda(k) \bigr\rrvert \times\llvert I \rrvert
\times P \bigl(\ta \bigl(\Lambda(k),\ell \bigr) \bigr)
\\
& \leq& d2^d \biggl(\frac{n}{k} \biggr)^{d-1} \bigl(
1+2d(2k+1)^{d-1} P \bigl(\ta \bigl(\Lambda(k),\ell \bigr) \bigr) \bigr)
\\
& \leq& c \biggl( \frac{n}{k} \biggr)^{d-1} +c n^{d-1} P
\bigl(\ta \bigl( \Lambda(k),\ell \bigr) \bigr),
\end{eqnarray*}
where $c$ is a constant depending on $d$ and $p$.
Plugging the inequality of
Corollary~\ref{intwoarms}
in the previous inequality, we get, with some larger constant $c$,
\[
E \bigl( \llvert {\mathcal C}\rrvert \bigr) \leq c \biggl(\frac{n}{k}
\biggr)^{d-1} + 
\frac{
cn^{d-1}k^{4d-2}
P (\ta(0,\ell-k) )}{
\inf\{
P ( a \longleftrightarrow b\mbox{ in }\Lambda(2k) )\dvtx a,b\in
\partial^{\mathrm{in}}\Lambda(k)
\}
}. %
\]
Noticing that
$E( \sqrt{\llvert {\mathcal C}\rrvert })\leq E( \llvert {\mathcal C}\rrvert )^{1/2}$,
we deduce from the central inequality stated in Lemma~\ref{ci}
and the previous inequality that
\begin{eqnarray*}
&& P \bigl(\ta(0,2n+\ell) \bigr)
\\
&&\qquad \leq
\frac{2d\ln n}{\sqrt{ \llvert \Lambda(n) \rrvert }} \biggl(c \biggl( \frac{n}{k} \biggr)^{d-1} +
\frac{cn^{d-1}k^{4d-2}
P (\ta(0,\ell-k) )}{
\inf\{
P ( a \longleftrightarrow b\mbox{ in }\Lambda(2k) )\dvtx a,b\in
\partial^{\mathrm{in}}\Lambda(k)
\}
} \biggr)^{1/2}
\\
&&\quad\qquad{}+ \frac{4d}{p(1-p)} \bigl\llvert \Lambda(n) \bigr\rrvert ^2
\exp \bigl(-{2} (\ln n)^2 p^2(1-p)^2
\bigr).
\end{eqnarray*}
We choose $\ell=n$, and we conclude that, for some constant $c$, we have
\begin{eqnarray*}
&&P \bigl(\ta(0,3n) \bigr)
\\
&&\qquad \leq
\frac{
c\ln n}{\sqrt{n}} \biggl(\frac{1}{k^{d-1}} + \frac{k^{4d-2}
P (\ta(0,n-k) )}{
\inf\{
P ( a \longleftrightarrow b\mbox{ in }\Lambda(2k) )\dvtx a,b\in
\partial^{\mathrm{in}}\Lambda(k)
\}
}
\biggr)^{1/2}.
\end{eqnarray*}
We shall next iterate this inequality in order to enhance the lower bound
on the two-arms exponent.
\section{Iterating at $p_c$}
In this section, we work at $p=p_c$ and we complete
the proofs of Theorems~\ref{fm} and~\ref{sm}.
Lemma~\ref{connect} yields that
\[
\forall k\geq1\qquad\inf \bigl\{ P \bigl( a \longleftrightarrow b\mbox { in }
\Lambda(2k) \bigr)\dvtx a,b\in\partial^{\mathrm{in}}\Lambda(k) \bigr\} \geq
\frac{c}{k^{2(d-1)d}}.
\]
From the last two inequalities, we deduce the following lemma.

\begin{lemma}
There exists $c>0$ such that,
for $1\leq k\leq n$,
\[
P \bigl(\ta(0,3n) \bigr) \leq\frac{
c
\ln n}{\sqrt{n}} \biggl(\frac{1}{k^{d-1}} +
k^{2d^2+2d-2} P \bigl(\ta(0,n-k) \bigr) \biggr)^{1/2}. %
\]
\end{lemma}

We shall next use iteratively the inequality of the lemma to improve
progressively the lower bound on the two arms exponent.
Suppose that
for some positive
constants $c',\beta,\gamma$, with $\gamma<1$, we have
%
\[
\forall n\geq2\qquad P \bigl(\ta(0,n) \bigr) \leq\frac{c'(\ln n)^\beta
}{n^\gamma}.
\]
Choosing $k=n^\delta$ with
\[
\delta= \frac{\gamma}{2d^2+3d-3},
\]
we obtain that
%
\[
\forall n\geq2\qquad P \bigl(\ta(0,3n) \bigr) \leq\frac{2c\sqrt
{c'}(\ln n)^{\beta/2+1}}{n^{\gamma'}},
\]
where
\[
\gamma' = \frac{1}{2}+\frac{d-1}{4d^2+6d-6}\gamma.
\]
%
By monotonicity,
\[
\forall n\geq3\qquad P \bigl(\ta(0,n) \bigr) \leq P \bigl(\ta \bigl(0,\lfloor
n/3 \rfloor \bigr) \bigr),
\]
therefore, there exists also a constant $c''$ such that
%
\[
\forall n\geq2\qquad P \bigl(\ta(0,n) \bigr) \leq\frac{c''(\ln
n)^{\beta+1}}{n^{\gamma'}}.
\]
The initial estimate stated in Proposition~\ref{ie}
yields that
\[
\forall n\geq2\qquad P \bigl(\ta(0,n) \bigr) \leq\frac{\kappa\ln
n}{\sqrt{n}}.
\]
We define a sequence of exponents $(\gamma_i)_{i\geq0}$ by setting
$\gamma_0={1}/{2}$ and
%
\[
\forall i\geq0\qquad\gamma_{i+1} = \frac{1}{2}+
\frac{d-1}{4d^2+6d-6}\gamma_i.
\]
Iterating the previous argument, we conclude that, for any $i\geq1$,
there exists
a constant $\alpha_i$ such that
%
\[
\forall n\geq2\qquad P \bigl(\ta(0,n) \bigr) \leq\frac{\alpha_i(\ln
n)^{i+1}}{n^{\gamma_i}}.
\]
It follows that
\[
\forall i\geq0\qquad\limsup_{n\to\infty} \frac{1}{\ln n} \ln P
\bigl(\ta(0,n) \bigr) \leq\gamma_i.
\]
The sequence
$(\gamma_i)_{i\geq0}$ converges geometrically toward
\[
\gamma_\infty= \frac{2d^2+3d-3}{4d^2+5d-5}.
\]
Letting $i$ go to $\infty$ in the previous inequality, we obtain
the result stated in Theorem~\ref{fm}.
Theorem~\ref{fm} and the inequality of
Corollary~\ref{intwoarms}
readily
imply Theorem~\ref{sm}.
To prove Theorem~\ref{sm},
we proceed as in
the proof of
Corollary~\ref{ftwoarms}, but instead of the initial estimate
of Proposition~\ref{ie}, we use the enhanced estimate
provided by Theorem~\ref{fm}.
\section{Proof of Theorem~\texorpdfstring{\protect\ref{tm}}{1.3}}
Throughout this section,
we work with a parameter~$p$
such that $\theta(p)>0$. We will use the hypothesis $\theta(p)>0$
to improve the lower bound for the probability of a connection inside
a finite box.

\begin{lemma}
\label{clb}
Let $n,\ell\geq2$.
For any
$x,y\in\Lambda(n)$, we have
\[
P \bigl(x\longleftrightarrow y\mbox{ in }\Lambda(n+\ell) \bigr) \geq
\theta(p)^2- P \bigl(\ta \bigl(\Lambda(n),x,y,\ell \bigr) \bigr).
\]
\end{lemma}

\begin{pf}
We write
\begin{eqnarray*}
&& P \bigl(x\longleftrightarrow y\mbox{ in }\Lambda(n+\ell) \bigr)
\\
&&\qquad \geq P
\pmatrix{
x\longleftrightarrow\partial^{\mathrm{in}}\Lambda(n+ \ell)\vspace*{2pt}\cr
y\longleftrightarrow\partial^{\mathrm{in}}\Lambda(n+\ell)\vspace*{2pt}\cr
x\longleftrightarrow y\mbox{ in }\Lambda(n+\ell)
}
\\
&&\qquad \geq P
\pmatrix{
x\longleftrightarrow\partial^{\mathrm{in}} \Lambda(n+\ell)\vspace*{2pt}\cr
y\longleftrightarrow\partial^{\mathrm{in}}\Lambda(n+\ell)
}
- P
\pmatrix{
x\longleftrightarrow \partial^{\mathrm{in}}\Lambda(n+\ell)\vspace*{2pt}\cr
y\longleftrightarrow\partial^{\mathrm{in}}\Lambda(n+\ell)\vspace*{2pt}\cr
x\hspace*{7pt}\not\hspace*{-7pt}\longleftrightarrow y\mbox{ in }\Lambda(n+\ell)
}.
\end{eqnarray*}
By the FKG inequality, we have
\[
P
\pmatrix{
x\longleftrightarrow\partial^{\mathrm{in}} \Lambda(n+\ell)\vspace*{2pt}\cr
y\longleftrightarrow\partial^{\mathrm{in}}\Lambda(n+ \ell)
}
\geq P
\pmatrix{
x\longleftrightarrow \infty\vspace*{2pt}\cr
y \longleftrightarrow\infty} \geq\theta(p)^2.
\]
Moreover,
\[
P \pmatrix{
x\longleftrightarrow\partial^{\mathrm{in}} \Lambda(n+\ell)\vspace*{2pt}\cr
y\longleftrightarrow\partial^{\mathrm{in}}\Lambda(n+ \ell)\vspace*{2pt}\cr
x\hspace*{7pt}\not\hspace*{-7pt}\longleftrightarrow y\mbox{ in }\Lambda(n+ \ell)}\leq P \bigl(\ta \bigl(\Lambda(n),x,y,\ell
\bigr) \bigr).
\]
The last two inequalities imply the inequality stated in the lemma.
\end{pf}

Since
\[
\theta(p) \leq P \bigl(0\longleftrightarrow\partial^{\mathrm{in}}\Lambda(n)
\bigr) \leq\sum_{x\in\partial^{\mathrm{in}}\Lambda(n)} P \bigl(0\longleftrightarrow x
\mbox{ in }\Lambda(n) \bigr),
\]
then there exists $x_n$ in $\partial^{\mathrm{in}}\Lambda(n)$ such that
\[
P \bigl(0\longleftrightarrow x_n\mbox{ in }\Lambda(n) \bigr) \geq
\frac{\theta(p)}{
\llvert \partial^{\mathrm{in}}\Lambda(n) \rrvert } \geq\frac{\theta(p)}{
2d(2n+1)^{d-1}
}. %
\]
We apply the inequality of Lemma~\ref{ineqab} to $0$ and $x_n$
with $k=0$:
\[
P \bigl( \ta \bigl(\Lambda(n),0,x_n,\ell \bigr) \bigr) \leq
\frac{
3^{4d}
}{p} n^{2d} \frac{P (\ta(0,\ell) )}{
P ( 0 \longleftrightarrow x_n\mbox{ in }\Lambda(n) )
}. %
\]
Combining the two previous inequalities, we conclude that
\[
P \bigl( \ta \bigl(\Lambda(n),0,x_n,\ell \bigr) \bigr) \leq
\frac{
3^{7d}
}{p\theta(p)} n^{3d-1} P \bigl(\ta(0,\ell) \bigr). %
\]
%
We apply the inequality of Lemma~\ref{clb} to $0$ and $x_n$,
and, together with the previous inequality, we obtain
\[
P \bigl(0\longleftrightarrow x_n\mbox{ in }\Lambda(n+\ell) \bigr)
\geq\theta(p)^2- \frac{
3^{7d}
}{p\theta(p)} n^{3d-1} P \bigl(\ta(0,
\ell) \bigr).
\]
Let $\alpha$ be such that
\[
\alpha> \frac{4d^2+5d-5}{2d^2+3d-3}(3d-1).
\]
We take $\ell=n^\alpha$.
By Theorem~\ref{fm}, for $n$ large enough,
\[
P \bigl(0\longleftrightarrow x_n\mbox{ in }\Lambda
\bigl(n+n^\alpha \bigr) \bigr) \geq\tfrac{1}{2}\theta(p)^2.
\]
Suppose, for instance, that $x_n$ belongs to $\{ n \}\times{\mathbb Z}^{d-1}$.
Let $e_1=(1,0,\ldots,0)$.
By symmetry and the FKG inequality, for $n$ large enough,
\begin{eqnarray*}
&& P \bigl(0\longleftrightarrow2ne_1\mbox{ in }\Lambda
\bigl(4n+n^\alpha \bigr) \bigr)
\\
&&\qquad \geq
P \pmatrix{
0\longleftrightarrow x_n\mbox{ in } \Lambda \bigl(n+n^\alpha \bigr)\vspace*{2pt}\cr
x_n \longleftrightarrow2ne_1 \mbox{ in } 2ne_1+\Lambda \bigl(n+n^\alpha \bigr)}
\\
&&\qquad \geq
P \bigl(0\longleftrightarrow x_n\mbox{ in }\Lambda
\bigl(n+n^\alpha \bigr) \bigr) P \bigl( x_n
\longleftrightarrow2ne_1\mbox{ in } 2ne_1+\Lambda
\bigl(n+n^\alpha \bigr) \bigr)
\\
&&\qquad \geq P \bigl(0\longleftrightarrow x_n\mbox{ in }\Lambda
\bigl(n+n^\alpha \bigr) \bigr)^2 \geq\frac{1}{4}
\theta(p)^4.
\end{eqnarray*}
Thus, there exists $N\geq1$ such that
\[
\forall n\geq N\qquad P \bigl(0\longleftrightarrow2ne_1\mbox{ in }
\Lambda \bigl(4n+n^\alpha \bigr) \bigr) \geq\tfrac{1}{4}
\theta(p)^4. %
\]
Let $n\geq N$ and let
$k\in\{ N,\ldots,n \}$. We have
%
\[
P \bigl(0\longleftrightarrow2ke_1\mbox{ in }\Lambda
\bigl(4n+n^\alpha \bigr) \bigr) \geq P \bigl(0\longleftrightarrow2ke_1
\mbox{ in }\Lambda \bigl(4k+k^\alpha \bigr) \bigr) \geq\tfrac{1}{4}
\theta(p)^4. %
\]
%
This implies further that
\[
\forall k\in\{ 2N,\ldots,2n \}\qquad P \bigl(0\longleftrightarrow
ke_1\mbox{ in }\Lambda \bigl(4n+n^\alpha \bigr) \bigr) \geq
\frac{p}{4}\theta(p)^4. %
\]
Since $N$ is independent of $n$, we conclude that
there exists $\rho>0$ such that
\[
\forall n\geq N, \forall k\in\{ 0,\ldots,2n \}\qquad P \bigl(0
\longleftrightarrow ke_1\mbox{ in }\Lambda \bigl(4n+n^\alpha
\bigr) \bigr) \geq\rho.
\]
Since $N$ is fixed, this lower bound can be extended to every $n\geq1$
by taking a smaller value of $\rho$.
By symmetry, we have the same lower bounds for the probabilities of connections
along the other axis directions. Using the FKG inequality, we conclude that
\[
\forall n\geq1,\forall x\in\Lambda(2n)\qquad P \bigl(0\longleftrightarrow x
\mbox{ in }\Lambda \bigl(6n+n^\alpha \bigr) \bigr) \geq
\rho^d.
\]
This completes the proof of Theorem~\ref{tm}.

\section*{Acknowledgements}
I thank Jeffrey Steif for telling me that quantitative estimates
can be derived from the arguments of \cite{AKN} or \cite{GGR}.
I thank an anonymous referee for his careful reading and his remarks,
which helped to improve the presentation.




%

\printaddresses
\end{document}